\documentclass[a4paper]{amsart}

\newtheorem{theorem}{Theorem}[section]
\newtheorem{lemma}[theorem]{Lemma}
\newtheorem{corollary}[theorem]{Corollary} 

\theoremstyle{definition}

\theoremstyle{remark}
\newtheorem{remark}[theorem]{Remark}

\numberwithin{equation}{section}

\newcommand{\indic}[1]{1_{\{#1\}}}
\newcommand{\TT}{\mathbb{T}} 
\newcommand{\EE}{\mathbb{E}} 
\newcommand{\PP}{\mathbb{P}} 
\newcommand{\RR}{\mathbb{R}} 
\newcommand{\CC}{\mathbb{C}} 
\newcommand{\hilbcircle}{\mathcal{H}^{\TT}}
\newcommand{\hilbline}{\mathcal{H}^{\RR}} 
\newcommand{\lone}{L^{1}(\TT)} 
\newcommand{\ltwo}{L^{2}(\TT)} 
\newcommand{\uonec}{\widetilde{U}} 
\newcommand{\utwo}{U_2} 
\newcommand{\utwoc}{\widetilde{\utwo}} 
\newcommand{\itilde}{{(I)}}
\newcommand{\iitilde}{{(II)}}
\newcommand{\iiitilde}{{(III)}}
\newcommand{\ivtilde}{{(IV)}}

\begin{document}

\title[Hilbert transform]{A note on sharp one-sided bounds\\
for the Hilbert transform}

\author{Micha{\l} Strzelecki}
\address{Department of Mathematics, Informatics and Mechanics, University of Warsaw, Banacha 2, 02-097 Warsaw, Poland}
\email{m.strzelecki@mimuw.edu.pl}

\subjclass[2010]{Primary: 31A05, 60G44. Secondary: 42A50, 42A61.}

\date{November 17, 2014.}

\keywords{Hilbert transform, martingale, differential subordination, weak-type
inequality, best constants.}

\begin{abstract}
Let $\mathcal{H}^{\mathbb{T}}$ denote the Hilbert transform on the circle. The paper contains the proofs of the sharp estimates
\begin{equation*}
\frac{1}{2\pi}|\{ \xi\in\mathbb{T} : \mathcal{H}^{\mathbb{T}}f(\xi) \geq 1 \}| \leq \frac{4}{\pi}\arctan\left(\exp\left(\frac{\pi}{2}\|f\|_1\right)\right) -1, \quad f\in L^{1}(\mathbb{T}),
\end{equation*}
and
\begin{equation*}
\frac{1}{2\pi}|\{ \xi\in\mathbb{T} : \mathcal{H}^{\mathbb{T}}f(\xi) \geq 1 \}| \leq \frac{\|f\|_2^2}{1+\|f\|_2^2}, \quad f\in L^{2}(\mathbb{T}).
\end{equation*}
Related estimates for orthogonal martingales satisfying a subordination condition are also established.\end{abstract}

\maketitle

\section{Introduction}
\label{section-introduction}

Let $\hilbcircle$ denote the Hilbert transform on the unit circle~$\TT$ defined as the singular integral
\begin{equation*}
\hilbcircle f(e^{it}) = \frac{1}{2\pi} \operatorname{p.v.} \int_{-\pi}^{\pi} f(e^{is}) \cot \frac{t-s}{2} d s, \quad f\in\lone.
\end{equation*}
It is not a bounded operator on $\lone$, but a classical result of Kolmogorov~\cite{kolmogorov} states that 
\begin{equation*}
\frac{1}{2\pi} \left|\left\{ \xi\in\TT : |\hilbcircle f(\xi)| \geq 1\right\}\right| \leq c_1 \|f\|_1, \quad f\in\lone,
\end{equation*}
for some universal constant $c_1$ (here and below the symbol $\|\cdot\|_p$, $p\in[1,2]$, denotes the $p$-th norm with respect to the normalized Haar measure on $\TT$).

Recently Os\c{e}kowski~\cite{osekowski-one-sided} considered the following one-sided version of this estimate:
\begin{equation}\label{inequality-osekowski}
\frac{1}{2\pi} \left|\left\{ \xi\in\TT : \hilbcircle f(\xi) \geq 1\right\}\right| \leq \|f\|_1, \quad f\in\lone.
\end{equation}
The above inequality is optimal: for every $c<1$ there exists an integrable function $f$ such that $|\{ \xi\in\TT : \hilbcircle f(\xi) \geq 1\}| > 2\pi c \|f\|_1$ (see \cite{osekowski-one-sided}).

Our motivation comes from the following question. Let $m\in[0,1]$ be a given number and suppose that a function $f\in\lone$ satisfies $ |\{ \xi\in\TT : \hilbcircle f(\xi) \geq 1 \}| = 2\pi m$. How small can $\|f\|_1$ be? (For similar problems, arising in the context of martingale
transforms and the Haar system, consult the works of Burkholder \cite{burkholder-boundary-value} and
Choi \cite{choi-kp}.) Clearly, inequality \eqref{inequality-osekowski} gives some initial insight into this problem:
we must have $\|f\|_1 \geq m$. However, this bound is not sharp (as will follow from Corollary \ref{corollary-p-1-functions} below).

To answer the above question, we establish a class of weak-type bounds.	

\begin{theorem}\label{theorem-p-1-functions}
For every  $c\in(0,1]$ and every function $f\in\lone$ we have
\begin{align}\label{inequality-p-1-functions}
\begin{split}
\frac{1}{2\pi}\left| \left\{ \xi\in\TT : \hilbcircle f (\xi) \geq 1 \right\} \right| \leq & c \| f \|_1 \\
& +  1-\frac{2}{\pi}\arcsin(c)- \frac{2c}{\pi}\ln\left(1/c+\sqrt{1/c^2-1}\right).
\end{split}
\end{align}
For every $c\in(0,1]$ the constant added on the right-hand side is optimal.
\end{theorem}

If we optimize the right-hand side of inequality \eqref{inequality-p-1-functions} with respect to the parameter $c\in(0,1]$, we can rewrite the above statement in the following way.

\begin{corollary}\label{corollary-p-1-functions}
For every function $f\in\lone$ we have
\begin{equation}\label{inequality-corollary-p-1-functions}
\frac{1}{2\pi}\left| \left\{ \xi\in\TT : \hilbcircle f (\xi) \geq 1 \right\} \right| \leq \frac{4}{\pi} \arctan \left( \exp\left( \frac{\pi}{2}\| f \|_1\right) \right) - 1.
\end{equation}
Moreover, for every  $m\in[0,1)$ there exists a function $f\in\lone$ for which both sides of \eqref{inequality-corollary-p-1-functions} are equal to $m$.
\end{corollary}

Equivalently: if a function $f\in\lone$ satisfies $ |\{ \xi\in\TT : \hilbcircle f(\xi) \geq 1 \}| = 2\pi m$, then
\begin{equation*}
\|f\|_1 \geq \frac{2}{\pi} \ln \left( \tan \left( \frac{\pi}{4}(m+1)\right)\right)
\end{equation*}
and the right-hand side cannot be improved.

The second result contained in this note is the following one-sided version of the weak-type $(2,2)$ inequality for the Hilbert transform.

\begin{theorem}\label{theorem-p-2-functions}
For every  $c\in[0,1]$ and every function $f\in\ltwo$ we have
\begin{equation}\label{inequality-p-2-functions}
\frac{1}{2\pi}\left| \left\{ \xi\in\TT : \hilbcircle f (\xi) \geq 1 \right\} \right| \leq c^2 \| f \|_2^2 +  (1-c)^2.
\end{equation}
For every $c\in[0,1]$ the constant added on the right-hand side is optimal.
\end{theorem}

\begin{corollary}\label{corollary-p-2-functions}
For every function $f\in\ltwo$ we have
\begin{equation}\label{inequality-corollary-p-2-functions}
\frac{1}{2\pi}\left| \left\{ \xi\in\TT : \hilbcircle f (\xi) \geq 1 \right\} \right| \leq \frac{\|f\|_2^2}{1+\|f\|_2^2}.
\end{equation}
Moreover, for every  $m\in[0,1)$ there exists a function $f\in\ltwo$ for which both sides of \eqref{inequality-corollary-p-2-functions} are equal to $m$.
\end{corollary}

As in \cite{osekowski-one-sided} we in fact establish more general statements for orthogonal martingales which satisfy a subordination condition. We postpone further details concerning the probabilistic setting to Section \ref{section-probabilistic}, where we introduce all necessary definitions and formulate the probabilistic counterparts of Theorems \ref{theorem-p-1-functions} and \ref{theorem-p-2-functions}.

In Section \ref{section-osekowskis-function} we recall the construction and properties of the special function which was used to prove inequality \eqref{inequality-osekowski}. Then in Section \ref{section-proof-p-1} we explain how to modify this function in order to prove inequality \eqref{inequality-p-1-functions}. The sharpness of this inequality and the proof of Corollary \ref{corollary-p-1-functions} are presented in Section \ref{section-sharpness-p-1}. An analogous discussion for inequality \eqref{inequality-p-2-functions} can be found in Section \ref{section-p-2}.

In Section \ref{section-applications} we apply the results to find one-sided versions of the weak-type $(p,q)$ bounds on the Hilbert transform for $0<q\leq p\in\{1,2\}$.

Finally,  in Section \ref{section-final-remarks} we give some remarks about  one-sided versions of the weak-type $(p,p)$ inequalities for $1<p<2$. We also explain why inequalities \eqref{inequality-p-1-functions} and \eqref{inequality-p-2-functions} do not transfer to analogous estimates for the Hilbert transform on the real line.

\section{Probabilistic setting}
\label{section-probabilistic}

Let $(\Omega, \mathcal{F}, \PP)$ be a complete probability space, filtered by $(\mathcal{F})_{t\geq 0}$, a nondecreasing
family of sub-$\sigma$-algebras of $\mathcal{F}$, such that $\mathcal{F}_0$ contains all the events of
probability $0$. Let $X = (X_t)_{t\geq 0}, Y= (Y_t)_{t\geq 0}$ be two adapted real martingales with continuous paths and let $[X, Y]$ denote their quadratic covariance process (see e.g. \cite{dellacherie-meyer} for details). We say that the processes $X$ and $Y$
are orthogonal if $[X, Y]$ is constant almost surely. We say that $Y$ is differentially subordinate to $X$ if the process
$([X, X]_{t} -   [Y, Y ]_{t} )_{t\geq 0 }$ is nondecreasing and nonnegative as a function of $t$.

We establish the following results.

\begin{theorem}\label{theorem-p-1-martingales}
Assume that $X, Y$ are orthogonal martingales such that $Y$ is differentially subordinate to $X$ and $Y_0=0$. Then, for every $c\in(0,1]$,
\begin{equation}\label{inequality-p-1-martingales}
\PP ( \sup_{\substack t \geq 0} Y_t \geq 1) \leq c \| X \|_1 +  1-\frac{2}{\pi}\arcsin(c)- \frac{2c}{\pi}\ln\left(1/c+\sqrt{1/c^2-1}\right).
\end{equation}
For every $c\in(0,1]$ the constant added on the right-hand side is optimal.
\end{theorem}

\begin{theorem}\label{theorem-p-2-martingales}
Assume that $X, Y$ are orthogonal martingales such that $Y$ is differentially subordinate to $X$ and $Y_0=0$. Then, for every $c\in[0,1]$,
\begin{equation}\label{inequality-p-2-martingales}
\PP ( \sup_{\substack t \geq 0} Y_t \geq 1) \leq c^2 \| X \|_2^2 + (1-c)^2 \end{equation}
For every $c\in[0,1]$ the constant added on the right-hand side is optimal.
\end{theorem}

The Reader will easily formulate the probabilistic counterparts of Corollaries \ref{corollary-p-1-functions} and \ref{corollary-p-2-functions}.

\section{Special function from \cite{osekowski-one-sided}}
\label{section-osekowskis-function}

Recall that the conformal mapping $K(z) =(\sqrt{z} - 1/{\sqrt{z}})/2$ maps the upper half-plane $H = \RR\times(0,\infty)$ onto the set $H \setminus \{ ai : a \geq 1\}$. The inverse of $K$ is given by the formula $L(z) = 2z^2+1+2z\sqrt{z^2+1}$ (here and below we set $\sqrt{z} = \sqrt{r}e^{i\theta / 2}$ if $z = re^{i\theta}$, $r\geq 0$, $\theta \in (-\pi, \pi]$). 

Introduce the function $\mathcal{U}:H\rightarrow\RR$ given by the Poisson integral
\begin{equation*}
\mathcal{U}(\alpha, \beta) = \frac{1}{\pi} \int_0^{\infty}\frac{\beta\left( 1 - \frac{1}{2}\left|\sqrt{t}-\sqrt{t^{-1}}\right|\right)}{(\alpha -t)^2+\beta^2} d t, \quad (\alpha, \beta)\in H.
\end{equation*}

The special function from \cite{osekowski-one-sided} which was used to prove \eqref{inequality-osekowski} is defined as follows: $U(x,y) = \mathcal{U}(L(x,y))$ for $(x,y)\in H\setminus \{ ai : a \geq 1\}$; $U(x,y) = 1-|x|$ for $(x,y)\in \RR \times(-\infty, 0]$;  $U(0,y) = 0$ for $y\geq 1$. The next lemma sums up its properties (for the proof see \cite[Section 3]{osekowski-one-sided}).

\begin{lemma}\label{lemma-osekowski-special-function}
Function $U:\RR^2\rightarrow\RR$ is continuous, symmetric with respect to the first variable: $U(x,y)=U(-x,y)$, and enjoys the following four properties.
\begin{enumerate}
\item[(i)] For any $x,y\in\RR$ we have $U(x,y)\geq\indic{y\leq 0} - |x|$.
\item[(ii)] For any $x\in\RR$ we have $U(x,1)\leq 0$.
\item[(iii)] For any $y\in\RR$ the function $U(\cdot,y): x\mapsto U(x,y)$ is concave on $\RR$.
\item[(iv)] $U$ is superharmonic.
\end{enumerate}
Moreover the function $(x,y)\mapsto U(x,y)+|x|$ is bounded.
\end{lemma}

We will also need one additional property of the special function $U$.

\begin{lemma}\label{lemma-value-u-0-c}
For $c\in(0,1]$ we have
\begin{equation*}
U(0,c) = 1-\frac{2}{\pi}\arcsin(c)- \frac{2c}{\pi}\ln\left(1/c+\sqrt{1/c^2-1}\right).
\end{equation*}
\end{lemma}

\begin{proof} For $c=1$ we see from the definition that $U(0,1) = 0$. Fix therefore a $c\in(0,1)$ and denote $A = 1-2c^2$, $B = 2c\sqrt{1-c^2}$. Then $L(0,c)=(A, B)$ and $A^2+B^2=1$. From the definition of $U$ in the upper half-plane 
\begin{align*}
U(0,c) =  \frac{\pi+2\arctan\left(\frac{A}{B}\right)}{2\pi} - \frac{B}{2\pi} \int_0^{\infty}\frac{\left|\sqrt{t}-\sqrt{t^{-1}}\right|}{(A-t)^2+B^2} d t.
\end{align*}
Notice that if we change the interval of integration in the remaining integral to $(1,\infty)$, then we will get exactly one half of the integral's value (since we can substitute $1/t$ instead of $t$ for $t\in(0,1)$ and use  $A^2+B^2=1$). We then substitute $t=s^2$, $1<t<\infty$, and what is left is to integrate a rational function.
\begin{align*}
\frac{B}{2\pi} \int_0^{\infty}\frac{\left|\sqrt{t}-\sqrt{t^{-1}}\right|}{(A-t)^2+B^2} d t & =  \frac{B}{\pi} \int_1^{\infty}\frac{ \sqrt{t}-\sqrt{t^{-1}}}{(A-t)^2+B^2} d t =	
\frac{2B}{\pi} \int_1^{\infty}\frac{s^2-1}{s^4-2As+1} d s \\
& =  \frac{B}{\pi\sqrt{2(1+A)}}\ln\left(\frac{ 2+\sqrt{2(1+A)}}{ 2-\sqrt{2(1+A)}}\right).
\end{align*}
Now we can substitute the values of $A$ and $B$ and use simple algebraic and trigonometric identities (remembering that $A^2+B^2=1$) to get 
\begin{align*}
U(0,c)  & =  \frac{\pi+2\arctan\left(\frac{1-2c^2}{2c\sqrt{1-c^2}}\right)}{2\pi} -  \frac{c}{\pi}\ln\left(\frac{ 1+\sqrt{1-c^2}}{1-\sqrt{1-c^2}}\right) \\
 & = 1-\frac{2}{\pi}\arcsin(c) -  \frac{2c}{\pi}\ln\left( 1/c+\sqrt{1/c^2-1}\right),
\end{align*} which is the assertion of the Lemma.
\end{proof}

\begin{remark} Lemma \ref{lemma-value-u-0-c}  allows us to brute-force check that function $U(0,\cdot)$ is convex on $[0,\infty)$ simply by calculating its second derivative (an observation from the Proof of Corollary \ref{corollary-p-1-functions} in the next section simplifies the computations). This provides a new proof of \cite[Lemma 3.2]{osekowski-one-sided} needed to prove property (iii) from Lemma \ref{lemma-osekowski-special-function} above.
\end{remark}

\section{Proofs of \eqref{inequality-p-1-martingales} and \eqref{inequality-p-1-functions}}
\label{section-proof-p-1}

\begin{proof}[Proof of \eqref{inequality-p-1-martingales}.] Fix any $c\in(0,1)$ ($c=1$ corresponds to inequality \eqref{inequality-osekowski} and the limit case $c=0$ to the trivial estimate of probability by $1$) and introduce the function $\uonec(x,y) = U(cx,cy)/c$. Due to Lemma \ref{lemma-osekowski-special-function} this function is continuous, symmetric with respect to the first variable, and enjoys the following properties.
\begin{enumerate}
\item[(i)] For any $x,y\in\RR$ we have $\uonec(x,y)\geq\frac{1}{c}\indic{y\leq 0} - |x|$.
\item[(ii)] For any $x\in\RR$ we have $\uonec(x,1)\leq \uonec(0,1) =  U(0,c)/c$.
\item[(iii)] For any $y\in\RR$ the function $\uonec(\cdot,y): x\mapsto \uonec(x,y)$ is concave on $\RR$.
\item[(iv)] $\uonec$ is superharmonic.
\end{enumerate}
We now mimic the proof of inequality \eqref{inequality-osekowski}.

First note that $\uonec$ is not smooth since  $\uonec(x,y) = 1/c-|x|$ for $y\leq 0$. Let $g:\RR^2\rightarrow [0,\infty)$ be a smooth radial function supported on a ball of center $(0,0)$ and radius $1$, satisfying $\int_{\RR^2} g = 1$. For $\delta>0$ define
\begin{equation*}
\uonec^{\delta}(x,y) = \int_{\RR^2} \uonec(x+\delta r, y+\delta s) g(r,s) d r d s.
\end{equation*}
Function $\uonec^{\delta}$ inherits properties (i) -- (iv) in a slightly changed form. For example, (i) implies that
\begin{align*}
\uonec^{\delta}(x,y) \geq & \int_{\RR^2} \indic{y+\delta s\leq 0} g(r,s) d r d s - \int_{\RR^2} |x+\delta r|g(r,s) d r d s\\
\geq & \indic{y\leq -\delta} - (|x| + \delta).
\end{align*}
We also have $\uonec^{\delta}(x,1)\leq\uonec(x,1)\leq U(0,c)/c$, since $\uonec$ is superharmonic and $g$ is radial. Properties (iii) and (iv) transfer directly to function $\uonec^{\delta}$. Summarizing, function $\uonec^{\delta}$ is smooth, symmetric with respect to the first variable, and enjoys the following properties.
\begin{enumerate}
\item[\itilde] For any $x,y\in\RR$ we have $\uonec^{\delta}(x,y)\geq\frac{1}{c}\indic{y\leq -\delta} -(|x|+\delta)$.
\item[\iitilde] For any $x\in\RR$ we have $\uonec^{\delta}(x,1)\leq U(0,c)/c$.
\item[\iiitilde] For any $y\in\RR$ the function $\uonec^{\delta}(\cdot,y): x\mapsto \uonec(x,y)$ is concave on $\RR$.
\item[\ivtilde] $\uonec^{\delta}$ is superharmonic.
\end{enumerate}

Let $X, Y$ be martingales as in the statement, localized if necessary in order to guarantee the integrability of all random variables below. Let $\tau = \inf\{t\geq 0 : Y_t\geq 1+\varepsilon\}$. The It\^{o}'s formula gives
\begin{equation*}
\uonec^{\delta}(X_{\tau\wedge  t},1-Y_{\tau\wedge  t}) = \uonec^{\delta}(X_0,1-Y_0) + I_1 + I_2/2,
\end{equation*}
where
\begin{align*}
I_1 = & \int_{0+}^{\tau\wedge  t} \uonec^{\delta}_x(X_s, 1- Y_s) d X_s - \int_{0+}^{\tau\wedge  t} \uonec^{\delta}_y(X_s,1- Y_s) d Y_s,\\
I_2 = & \int_{0+}^{\tau\wedge  t} \uonec^{\delta}_{xx}(X_s, 1- Y_s) d [X]_s\\
 & - 2\int_{0+}^{\tau\wedge  t} \uonec^{\delta}_{xy}(X_s,1- Y_s) d [X, Y]_s + \int_{0+}^{\tau\wedge  t} \uonec^{\delta}_{yy}(X_s, 1- Y_s) d [Y]_s.
\end{align*}
We will estimate the expected values of the above sums separately.

We have $Y_0=0$, so  \iitilde{} implies that $\uonec^{\delta}(X_0,1-Y_0) \leq U(0,c)/c$. Moreover $\EE I_1 = 0 $, since the stochastic integrals in this sum are martingales. The middle term in $I_2$ vanishes, because $X$ and $Y$ are orthogonal. Using $\uonec^{\delta}_{xx}\leq 0$ (which follows from \iiitilde{}), the subordination of $Y$ to $X$, and property \ivtilde{} we see that $I_2\leq 0$.

After taking the expected value of both sides we arrive at $\EE \uonec^{\delta}(X_{\tau\wedge  t},1-Y_{\tau\wedge  t}) \leq U(0,c)/c$. Hence, by  \itilde{},
\begin{equation*}
\PP(Y_{\tau\wedge  t} \geq 1+\delta) \leq c\EE (|X	_{\tau\wedge  t}| +\delta) + U(0,c).
\end{equation*}

Letting $\delta\rightarrow 0$ we see that $\PP(\sup_{t\geq 0} Y_{\tau\wedge t} > 1) \leq c\EE |X_{\tau\wedge  t}| + U(0,c)\leq c\|X\|_1 + U(0,c)$. Therefore
\begin{equation*}
\PP(\sup_{t\geq 0} Y_{t} \geq 1+2\varepsilon) \leq \lim_{\substack t\rightarrow \infty} \PP(\sup_{t\geq 0} Y_{\tau\wedge t} \geq 1) \leq c\|X\|_1 + U(0,c).
\end{equation*}
After applying this bound to $X/(1+2\varepsilon)$, $Y/(1+2\varepsilon)$ and then letting $\varepsilon\rightarrow 0$ we finally conclude that
\begin{equation*}
\PP(\sup_{t\geq 0} Y_{t} \geq 1) \leq c\|X\|_1 + U(0,c).
\end{equation*}
By Lemma \ref{lemma-value-u-0-c} this finishes the proof of \eqref{inequality-p-1-martingales}.
\end{proof}

\begin{proof}[Proof of \eqref{inequality-p-1-functions}.]
Let $B$ be a planar Brownian motion starting from $0\in\CC$ and let $\tau=\inf\{t\geq 0 : |B_t| = 1 \}$. Denote the harmonic extensions of $f$ and $\hilbcircle f$ to the unit disk by $u$ and $v$ respectively. They satisfy the Cauchy-Riemann equations and $v(0)=0$. It follows from the It\^{o}'s formula that the martingales $X= (u(B_{\tau\wedge t}))_{t\geq 0}$ and $Y= (v(B_{\tau\wedge t}))_{t\geq 0}$ are orthogonal, $Y$ is differentially subordinate to $X$ and $Y_0 = 0$. Moreover $B_{\tau}$ is distributed uniformly on the unit circle. Therefore the martingale inequality \eqref{inequality-p-1-martingales} applied to $X$ and $Y$ gives us inequality \eqref{inequality-p-1-functions}.
\end{proof}

\section{Sharpness of \eqref{inequality-p-1-functions} and \eqref{inequality-p-1-martingales}, proof of Corollary \ref{corollary-p-1-functions}}
\label{section-sharpness-p-1}

\begin{proof}[Sharpness of \eqref{inequality-p-1-martingales}.] Let $X=(X_t)_{t\geq 0}$, $Y = (Y_t)_{t\geq 0}$ be such that $(X,1-Y)$ is a planar Brownian motion starting from point $(0,1)$ and killed upon hitting the boundary of the set $H \setminus \{ ai : a \geq 1/c\}$. Then $X$, $Y$ are orthogonal martingales, $Y$ is subordinate to $X$ and $Y_0=0$.

Introduce the stopping time $\tau = \inf\{t>0 : 1-Y_t = 0\}$.  
The process $(\uonec(X_t, 1-Y_t))_{t\geq 0}$  is a martingale of mean $\uonec(0,1) = U(0, c)/c > 0$ and hence
\begin{align*}
U( 0,c)  & = c\EE{}  \uonec(X_t, 1-Y_t) = \EE \left[(1 - c|X_{t}|)\indic{\tau \leq  t} +  U(c X_{ t} , c(1-Y_t)) \indic{ \tau > t}\right] \\
 & =  \PP(\tau \leq  t) - c\EE|X_{t}| + \EE\left( U(  c X_{ t},  c (1-Y_{ t})) +c|X_{ t}| \right)\indic{\tau> t }.
\end{align*}
We now let $t\rightarrow\infty$. On the set $\{ \tau = \infty\}$ we almost surely have $\lim_{\substack t\rightarrow\infty} X_t = 0$ and $\lim_{\substack t\rightarrow\infty} c(1-Y_t)\geq 1$, hence $\lim_{\substack t\rightarrow\infty} U(cX_t,c(1-Y_t)) + c|X_t| = 0$. From Lemma \ref{lemma-osekowski-special-function} we know that the function $(x,y)\mapsto U(x,y) + |x|$ is bounded, so by the dominated convergence theorem the third term in the above sum vanishes. Moreover the expected values $\EE |X_t|$ converge monotonically to $\|X\|_1$ since $X$ is a martingale. Hence
\begin{equation*}
U(0,c) = \PP(\tau <\infty) - c\|X\|_1 =\PP(\sup_{\substack t\geq 0} Y_t \geq 1) - c\|X\|_1.
\end{equation*}
This (combined with Lemma \ref{lemma-value-u-0-c}) ends the proof of sharpness of \eqref{inequality-p-1-martingales}.
\end{proof}

\begin{remark} In the above example we can calculate $\PP(\sup_{\substack t\geq 0} Y_t \geq 1)$ and $\|X\|_1$ explicitly. We know that $z\mapsto L(cz)$ maps the set $H\setminus \{ ai : a \geq 1/c\}$  to the upper half-plane. The expression $\PP(\sup_{\substack t\geq 0} Y_t \geq 1)$ is equal to the probability that  the Brownian motion starting from point $L(0,c) = (1-2c^2,2c\sqrt{1-c^2})$ and killed upon hitting the boundary of $H$ will terminate on $(0,\infty)\times\{0\}$ and is therefore equal to 
\begin{align*}
\frac{\pi+2\arctan\left(\frac{1-2c^2}{2c\sqrt{1-c^2}}\right)}{2\pi} = 1-\frac{2}{\pi}\arcsin(c)
\end{align*} (see e.g. \cite{ransford} and compare with the calculations in the proof of Lemma \ref{lemma-value-u-0-c}). Hence
\begin{align*}
\|X\|_1 = \frac{U(0,c) - (1-2\arcsin(c)/\pi)}{c} = \frac{2}{\pi}\ln\left( 1/c+\sqrt{1/c^2-1}\right).
\end{align*}
\end{remark}

\begin{proof}[Sharpness of \eqref{inequality-p-1-functions}.] 
Consider a conformal mapping $N : H\setminus\{ai : a\geq 1/c\}\rightarrow D(0,1)$ which maps the set $H\setminus\{ai : a\geq 1/c\}$ onto the open unit disk and satisfies $N(0,1) = (0,0)$ ($N$ is the composition of the mapping $z\mapsto L(cz)$, which maps the set $H\setminus\{ai : a\geq 1/c\}$ onto the upper half-plane, and the homography  
\begin{equation*}
z\mapsto -2\left(\frac{z-\operatorname{Re} L(ci)}{\operatorname{Im} L(ci)} +i\right)^{-1}-i,
\end{equation*}
which maps the upper half-plane onto the unit disk). We define two conjugate harmonic functions: for $r\in[0,1), \ t\in(-\pi, \pi]$ we set $u(re^{it}) = \operatorname{Re} \ N^{-1} (re^{it})$, $v(re^{it}) = 1-\operatorname{Im} N^{-1} (re^{it})$. Functions $u, v$ have radial limits almost surely and the Hilbert transform of the function
\begin{equation*}
f(e^{it}) = \lim_{\substack r\rightarrow 1^{-}} u(re^{it})
\end{equation*}
is the function
\begin{equation*}
g(e^{it}) = \lim_{\substack r\rightarrow 1^{-}} v(re^{it}).
\end{equation*}
Moreover, $| \{ \xi\in\TT : g (\xi) \geq 1 \} | /(2\pi) =  \PP(\sup_{\substack t\geq 0} Y_t \geq 1)$ and $\| f \|_1 = \|X\|_1$, where $X$, $Y$ are like in the above proof of sharpness of \eqref{inequality-p-1-martingales}. This finishes the proof of sharpness of \eqref{inequality-p-1-functions}.
\end{proof}

\begin{proof}[Proof of Corollary \ref{corollary-p-1-functions}.] For $c\in(0,1)$ denote $P(c) = 1-2\arcsin(c)/\pi$, $E(c) =2 \ln (1/c+\sqrt{1/c^2-1})/\pi$, so that $U(0,c) = P(c) - cE(c)$. A bit lengthy computation reveals that we identically have $P'(c) - cE'(c)=0$. Hence, for a given function $f\in\lone$ the derivative of the right-hand side of \eqref{inequality-p-1-functions} with respect to $c\in[0,1]$ is equal to zero whenever $\|f\|_1 -E(c) = 0$, that is if $c = 2\exp(\frac{\pi}{2}\|f\|_1)/(1 + \exp(\pi\|f\|_1))$. Substituting such $c$ into \eqref{inequality-p-1-functions} reduces it to 
\begin{equation*}
\frac{1}{2\pi}\left| \left\{ \xi\in\TT : \hilbcircle f (\xi) \geq 1 \right\} \right|\leq c\|f\|_1 + P(c) - cE(c) = P(c),
\end{equation*}
which can be further simplified to 
\begin{equation*}
\frac{1}{2\pi}\left| \left\{ \xi\in\TT : \hilbcircle f (\xi) \geq 1 \right\} \right|\leq \frac{4}{\pi} \arctan \left(\exp\left(  \frac{\pi}{2}\|f\|_1\right)\right) -1.
\end{equation*}
As for sharpness, it is enough to take $f$ which is optimal in \eqref{inequality-p-1-functions} for the value of $c$ we have chosen.
\end{proof}

\section{Sketch of proofs of \eqref{inequality-p-2-martingales}, \eqref{inequality-p-2-functions} and \eqref{inequality-corollary-p-2-functions}}
\label{section-p-2}

The proofs of  inequalities \eqref{inequality-p-2-martingales}, \eqref{inequality-p-2-functions} and \eqref{inequality-corollary-p-2-functions} are analogous to those of inequalities \eqref{inequality-p-1-martingales}, \eqref{inequality-p-1-functions} and \eqref{inequality-corollary-p-1-functions}, but we have to exploit properties of a different special function. Consider the function $\utwo:\RR^2\rightarrow\RR$ defined by the formula
\begin{equation*}
 \utwo(x,y) =
  \begin{cases}
   1- x^2 			& \text{if } y \geq 0, \\
   (1-y)^2- x^2	& \text{if } 0 < y\leq 1, \\
   -x^2	     		& \text{if } y > 1.
  \end{cases}
\end{equation*}
 It is continuous and clearly enjoys the following properties.
\begin{enumerate}
\item[(i)] For any $x,y\in\RR$ we have $\utwo(x,y)\geq\indic{y\leq 0} - x^2$.
\item[(ii)] For any $x\in\RR$ we have $\utwo(x,1)\leq 0$.
\item[(iii)] For any $y\in\RR$ the function $\utwo(\cdot,y): x\mapsto U(x,y)$ is concave on $\RR$.
\item[(iv)] $\utwo$ is superharmonic.
\end{enumerate}
To prove inequality \eqref{inequality-p-2-martingales} we have to use the function $\utwoc:\RR^2\rightarrow\RR$ defined as $\utwoc (x,y) = \utwo(cx,cy)/c^2$ and follow the reasoning from Section \ref{section-proof-p-1}.

As for sharpness, let $X=(X_t)_{t\geq 0}$, $Y = (Y_t)_{t\geq 0}$ be such that $(X,1-Y)$ is a planar Brownian motion starting from point $(0,1)$ and killed upon hitting the boundary of the strip $\{ (x,y) : 0 \leq y \leq 1/c\}$. Then $X$, $Y$ are orthogonal martingales, $Y$ is subordinate to $X$ and $Y_0=0$.   Similarly as in Section \ref{section-sharpness-p-1}, it is enough to apply the Doob's optional sampling theorem to the martingale $(\utwoc(X_t, 1-Y_t))_{t\geq 0}$ which has mean $\utwoc(0,1) = \utwo(0, c)/c^2 = (1-c)^2/c^2 > 0$.

The calculations which occur during the proof of Corollary \ref{corollary-p-2-functions} are also straightforward. We leave the details to the Reader.

\section{An application}\label{section-applications}

For $0<q\leq 1$ and $0<q'\leq 2$ define
\begin{align*}
c(1,q) & = \sup_{\substack x>0} \frac{\left(\frac{4}{\pi}\arctan\left(\exp\left(\frac{\pi}{2} x \right)\right)  - 1\right)^{1/q}}{x},\\
c(2,q') & = \sup_{\substack x>0} \frac{1}{x}\left(\frac{x^2}{1+x^2}\right)^{1/q'} = \left(\frac{q'}{2}\right) ^{1/q}\left(\frac{2}{q'} - 1\right)^{1/q'-1/2}.
\end{align*}

\begin{corollary}\label{corollary-p-1-q-functions} For $0<q\leq 1$ and for every function $f\in \lone$ we have
\begin{equation}\label{inequality-corollary-p-1-q-functions}
\left(\frac{1}{2\pi}\left| \left\{ \xi\in\TT : \hilbcircle f (\xi) \geq 1 \right\} \right| \right)^{1/q} \leq c(1,q) \|f\|_1.
\end{equation}
\end{corollary} 

\begin{corollary}\label{corollary-p-2-q-functions} For $0<q\leq 2$  and for every function $f\in \ltwo$ we have
\begin{equation}\label{inequality-corollary-p-2-q-functions}
\left(\frac{1}{2\pi}\left| \left\{ \xi\in\TT : \hilbcircle f (\xi) \geq 1 \right\} \right| \right)^{1/q} \leq c(2,q) \|f\|_2.
\end{equation}
\end{corollary} 

\begin{remark} Both above inequalities are sharp. Moreover, for $q>1$ inequality \eqref{inequality-corollary-p-1-q-functions} does not hold with any constant $c(1,q) < \infty$, since we know from Section \ref{section-sharpness-p-1} that there exist functions for which both sides of the inequality \eqref{inequality-osekowski} tend to zero, but their quotient tends to one. Similarly, it makes no sense to consider \eqref{inequality-corollary-p-2-q-functions} for $q>2$.
\end{remark}

\begin{proof}[Proof of Corollary \ref{corollary-p-1-q-functions}.] Using Corollary \ref{corollary-p-1-functions} and the definition of $c(1,q)$ we get
\begin{align*}
\left(\frac{1}{2\pi}\left| \left\{ \xi\in\TT : \hilbcircle f (\xi) \geq 1 \right\} \right|\right)^{1/q} &\leq \left(\frac{4}{\pi}\arctan\left(\exp\left(\frac{\pi}{2} \|f\|_1 \right)\right)  - 1\right)^{1/q}\\
&\leq c(1,q) \|f\|_1.
\end{align*}
Moreover the function
\begin{equation*}
x\mapsto \frac{\left(\frac{4}{\pi}\arctan\left(\exp\left(\frac{\pi}{2} x \right)\right)  - 1\right)^{1/q}}{x},\quad x>0,
\end{equation*}
attains its maximum at some  positive $x=x_1$. By Corollary \ref{corollary-p-1-functions}  there exists a function $f\in\lone$ for which $\|f\|_1=x_1$ and
\begin{align*}
\frac{1}{2\pi}\left| \left\{ \xi\in\TT : \hilbcircle f (\xi) \geq 1 \right\} \right| =
\frac{4}{\pi}\arctan\left(\exp\left(\frac{\pi}{2} x_1	 \right)\right)  - 1.\end{align*}
This finishes the proof of optimality of the obtained inequality.\end{proof}

The proof Corollary \ref{corollary-p-2-q-functions} follows exactly the same lines and therefore we skip it.

\section{Final remarks}
\label{section-final-remarks}

We start this section by formulating an open problem: for a given $p\in(1,2)$ identify the best constant $c(p)$ in an $L^p$-version of \eqref{inequality-osekowski}
\begin{equation}\label{inequality-p-p-functions}
\frac{1}{2\pi}|\{ \xi\in\TT : \hilbcircle f(\xi) \geq 1 \}| \leq c(p) \|f\|_p^p.
\end{equation}
It would be even more interesting to find the sharp counterpart of estimate \eqref{inequality-corollary-p-1-functions}, that is to determine for each $p\in(1,2)$ and any positive $c$ (smaller than $c(p)$) the best constant $d(p,c)$, such that
\begin{equation*}
\frac{1}{2\pi}|\{ \xi\in\TT : \hilbcircle f(\xi) \geq 1 \}| \leq c \|f\|_p^p +d(p,c).
\end{equation*}

The author believes that for $p\in (1, 2)$ the value of $c(p)$ is $1$, but he was unable to solve the above problems. Nevertheless, we will  try to shed some light on them.

For $p\in[1,2]$ and $x\geq 0$ define
\begin{equation*}
R_p(x) = \sup \left\{ \frac{1}{2\pi}|\{ \xi\in\TT : \hilbcircle f(\xi) \geq 1 \}| \ : \ \|f\|^p_p\leq x\right\}.\end{equation*}
In the previous sections we have shown that $R_1(x) = \frac{4}{\pi}\arctan\exp(\frac{\pi}{2}x) - 1$ and $R_2(x) = x/(1+x)$. One can moreover check that $R_1\geq R_2$ (for example by comparison of derivatives of both sides).

If we knew that for any given $x\geq 0$ the function $p\mapsto R_p(x)$ is nonincreasing  with respect to the parameter $p\in[1,2]$,  then we could conclude that $c(p)=1$ for $p\in[1,2]$. Indeed, from $x\geq R_1(x)\geq R_p(x)$ we would get $c(p)\leq 1$. Moreover, from $R_p\geq R_2$ and the fact $R_2$ that is strictly convex and tangent to the linear function with slope $1$ it would follow that we cannot have $c(p)< 1$.

\begin{remark} For $x\geq 1$ the function $p\mapsto R_p(x)$ is monotonous. Indeed, for $ p\in[1,2] $ and $x\geq 1$ we have
\begin{equation*}
R_1(x) \geq R_1(x^{1/p}) \geq R_p(x) \geq R_2 ( x^{2/p}) \geq R_2(x),
\end{equation*}
where we used monotonicity of $R_1, R_2$ in the first and last inequality, and Jensen's inequality (if $\|f\|_2^2\geq x^{2/p}$, then $\|f\|_p^p\leq x$; if $\|f\|_p^p\leq x$, then $\|f\|_1\leq x$) in the two middle passages.
Unfortunately, in order to find $c(p)$ it is crucial to control $R_p(x)$ for small $x$.
\end{remark}

\begin{remark} We would be able to prove inequality \eqref{inequality-p-p-functions} if we found a
superharmonic function $U_p$ on the plane, satisfying
\begin{equation*}
U_p(x,y) \geq \indic{y\geq 0} - c(p) |x|^p
\end{equation*}
and $U_p(0,1)\leq 0$. Recall that for $p = 1$ the function was equal
to the function $(x, y)\mapsto \indic{y\geq 0} - |x|$ on the set $\{(x,y) : y\geq 0 \}\cup \{ (0,y) : y\geq 1 \}$  and was harmonic on the complement of this set. For $p = 2$ the
function $U_2$  was equal
to the function $(x, y)\mapsto \indic{y\geq 0} - |x|^2$ on the set $\{(x,y) : y\geq 0 \}\cup \{ (x,y) : y\geq 1 \}$  and was harmonic on the set $\{(x,y) : 0<y<1 \}$. This suggests that for
each $p$ there should be a domain $\Omega_p$  contained in $\{(x,y) : y > 0\}$ such that $U_p$ is harmonic on $\Omega_p$ and equal to the function $(x,y)\mapsto\indic{y\geq 0} - c(p) |x|^p$ outside $\Omega_p$. For a similar phenomenon arising in the context of martingale transforms,
where the corresponding problems were completely solved, see the works
of Os\c{e}kowski \cite{osekowski-survey} and Suh \cite{suh}. 
\end{remark}

The last remark addresses the problem of extending the results of the preceding sections to the Hilbert transform on the real line.

\begin{remark} The problem investigated in this note does not make sense in the nonperiodic setting. By a ``blowing-up the circle'' argument inequality \eqref{inequality-osekowski} implies the following sharp estimate 
for the Hilbert transform on the real line
\begin{equation}\label{inequality-line-osekowski}
\left| \left\{ x\in\RR : \hilbline f (x) \geq 1 \right\} \right| \leq \| f \|_1, \quad f\in L^{1}(\RR) 
\end{equation}
(see \cite{osekowski-one-sided} for details).
Suppose that for some positive $c$ there exists a finite constant $D(c)$ for which
\begin{align}\label{inequality-line-p-1-functions}
\left| \left\{ x\in\RR : \hilbline f (x) \geq 1 \right\} \right| \leq c\| f \|_1 +D(c)
\end{align}
holds for every integrable function $f:\RR\to\RR$. We will show that then $c$ has to be greater or equal than $1$ (in which case we can simply take $D(c) = 0$). For  $f\in L^{1}(\RR)$ and $\lambda>0$ define $f_{\lambda} (x) = f(x/\lambda)$. Then $\|f_\lambda\|_1 = \lambda \|f\|_1$, $\hilbline f_{\lambda}(x) = \hilbline f (x/\lambda)$ and 
\begin{align*}
\left| \left\{ x\in\RR : \hilbline f_{\lambda} (x) \geq 1 \right\} \right| & = \left| \left\{ x\in\RR : \hilbline f (x/\lambda) \geq 1 \right\} \right| \\
& = \lambda \left| \left\{ x\in\RR : \hilbline f (x) \geq 1 \right\} \right|.
\end{align*}
After applying  \eqref{inequality-line-p-1-functions} to $f_{\lambda}$, dividing both sides by $\lambda$ and letting $\lambda\to\infty$, we arrive at
\begin{align*}
\left| \left\{ x\in\RR : \hilbline f (x) \geq 1 \right\} \right| \leq c\| f \|_1.
\end{align*}
Our claim follows from the arbitrariness of $f\in L^{1}(\RR)$ and the sharpness of \eqref{inequality-line-osekowski}.
\end{remark}

\section*{Acknowledgments}

I would like to thank Adam Os\c{e}kowski for suggesting the problem and for many hours of conversations on related topics.

\bibliographystyle{amsplain}
\bibliography{strzelecki-one-sided-arxiv}

\end{document}